\documentclass[11pt, english]{article}

\usepackage[margin= 2 cm]{geometry}
\usepackage[T1]{fontenc}
\usepackage{amsthm}
\usepackage{amsmath}
\usepackage{amssymb}
\usepackage{setspace}
\usepackage{mathtools}
\usepackage{verbatim}
\usepackage{booktabs}
\usepackage{graphicx}
\usepackage[hidelinks]{hyperref}
\usepackage{cleveref}
\usepackage{graphicx}
\usepackage{appendix}
\usepackage[inline]{enumitem}
\usepackage{framed}
\usepackage{subcaption}
\usepackage{microtype,xurl}
\usepackage{caption}
\usepackage{listings}
\usepackage{xcolor}

\usepackage{floatrow}
\usepackage[T1]{fontenc}

\lstdefinestyle{verifier}{
  language=Python,
  basicstyle=\ttfamily\fontsize{8.1}{10.2}\selectfont,
  keywordstyle=\bfseries,
  commentstyle=\itshape,
  stringstyle=\ttfamily,
  showstringspaces=false,
  keepspaces=true,
  columns=fullflexible,
  breaklines=true,
  breakatwhitespace=false,
  numbers=left,
  numberstyle=\tiny,
  numbersep=7pt,
  xleftmargin=1.6em,
  aboveskip=0.8em,
  belowskip=0.8em,
  tabsize=4,
  upquote=true
}

\theoremstyle{plain}

\newtheorem*{thm*}{Theorem}
\newtheorem{thm}{Theorem}
\Crefname{thm}{Theorem}{Theorems}

\newtheorem*{lem*}{Lemma}
\newtheorem{lem}[thm]{Lemma}
\Crefname{lem}{Lemma}{Lemmas}

\newtheorem*{claim*}{Claim}

\Crefname{claim}{Claim}{Claims}

\newtheorem{prop}[thm]{Proposition}
\Crefname{prop}{Proposition}{Propositions}

\newtheorem{cor}[thm]{Corollary}
\Crefname{cor}{Corollary}{Corollaries}

\Crefname{conj}{Conjecture}{Conjectures}

\Crefname{qn}{Question}{Questions}

\Crefname{obs}{Observation}{Observations}

\newtheorem{ex}[thm]{Example}
\Crefname{ex}{Example}{Examples}

\theoremstyle{definition}

\Crefname{prob}{Problem}{Problems}

\Crefname{defn}{Definition}{Definitions}

\newtheorem*{defn*}{Definition}

\theoremstyle{remark}

\expandafter\def\expandafter\normalsize\expandafter{%
    \normalsize
    \setlength\abovedisplayskip{8pt}
    \setlength\belowdisplayskip{8pt}
    \setlength\abovedisplayshortskip{4pt}
    \setlength\belowdisplayshortskip{4pt}
}

\usepackage[square,sort,comma,numbers]{natbib}
 \setlist[itemize]{leftmargin=*}

\newcommand{\im}{\operatorname{im}}

 \newcommand{\R}{\mathbb R}
 \newcommand{\one}{\mathbf1}
 \newcommand{\file}[1]{\mbox{\nolinkurl{#1}}}

\newcommand{\Q}{\mathbb Q}
\newcommand{\Z}{\mathbb Z}

\newcommand{\Span}{\operatorname{span}}

\newcommand{\dist}{\operatorname{dist}}

\newcommand{\normZ}[1]{\left\|#1\right\|_{\R/\Z}}

\newcommand{\Ucal}{\mathcal U}
\newcommand{\Pcal}{\mathcal P}

\newcommand{\sgn}{\operatorname{sgn}}
\newcommand{\norm}[1]{\left\lVert #1\right\rVert_2}

\title{Hadwiger--Nelson Problem for Typical Norms}
\author{Noga Alon\thanks{Department of Mathematics, Princeton University and Schools of Mathematics, Computer Science and AI,
Tel Aviv University, Israel. Research supported in part by NSF grant DMS-2553988 and by AFOSR grant
26RT0135. Email: \href{nalon@math.princeton.edu} \textbf{nalon@math.princeton.edu}} \and Matija Buci\'c\thanks{Faculty of Mathematics, University of Vienna, Vienna, Austria. This research was funded in whole or in part by the Austrian Science Fund (FWF) 10.55776/AST2416125. Email: \href{mailto:matija.bucic@univie.ac.at} \textbf{matija.bucic@univie.ac.at}. } \and James Davies\thanks{Institute of Mathematics, Leipzig University, Augustusplatz 10, 04109 Leipzig, Germany. Email: \href{mailto:jgdavies@uwaterloo.ca} \textbf{jgdavies@uwaterloo.ca}. Supported by the Alexander von Humboldt Foundation in the framework of the Alexander von Humboldt Professorship of Daniel Král' endowed by the Federal Ministry of Education and Research.}}  
\date{}

\begin{document}

\maketitle

\begin{abstract}
The classical Hadwiger-Nelson problem asks for the chromatic number of the unit distance graph of the Euclidean plane. Over the years, this problem has been considered for a variety of other normed spaces, with higher-dimensional Euclidean space $\R^d$ being perhaps the most natural and well-studied.   
Alon, Buci\'c, and Sauermann proved that, for a typical norm on
$\R^d$, the chromatic number of the unit distance graph is at most $2^d$.
We improve this exponential bound to a linear one by showing that the chromatic number of a typical norm on $\mathbb{R}^d$ is at most $2d$ and that this is tight. This shows a stark difference in the behavior compared to the Euclidean case, where there is an exponential lower bound. 

One of the key ingredients is a certain high-dimensional, matrix generalization of the Lonely runner conjecture, which also allows us to completely settle the so-called view-obstruction conjecture of Schoenberg from 1978 and a more recent covering-radius conjecture of Henze and Malikiosis.
\end{abstract}

\section{Introduction}
The question of determining the chromatic number of the unit distance graph of the Euclidean norm in the plane is one of the most famous open problems in discrete geometry, known as the Hadwiger--Nelson problem. In other words, this question asks how many colors are needed in order to color all points in the plane such that any two points with Euclidean distance one receive different colors. This problem dates back to 1950, and it has been known for a long time that the answer is at least $4$ and at most $7$. In a recent breakthrough \cite{Exoo-Ismailescu,de-Grey}, the lower bound has been improved to $5$, sparking a collaborative Polymath project \cite{polymath_project} focusing on the problem. See also \cite{BMP-survey,Sw} for more details on the history of the Hadwiger--Nelson problem.

In general, for a norm $\|\cdot\|$ on $\R^d$, let $G_{\|\cdot\|}$ be the graph
with vertex set $\R^d$ in which $p$ and $p'$ are adjacent precisely
when $\|p-p'\|=1$. We refer to this graph as the \emph{unit-distance} graph of $\|\cdot\|$.
The Hadwiger-Nelson problem for other planar norms was first studied by Chilakamarri \cite{Chilakamarri} in 1991, who showed that the unit distance graph of every norm on $\R^2$ has chromatic number at least $4$ and at most $7$. 

In higher dimension $d$, it is known that the chromatic number of the unit distance graph of any $d$-norm is at most exponential in $d$ (see \cite{FK,Ku}). See also \cite{Sw} for a more detailed survey on what is known surrounding this problem. We also note that the results of Frankl and Wilson \cite{FW} give an exponential lower bound for the chromatic number of the unit distance graph of all $d$-norms which are invariant under coordinate permutations. So for many natural norms, the Hadwiger-Nelson problem has an answer growing exponentially with $d$. It is a very natural question to ask whether this behavior is ``typical'' and in fact whether it is even possible to have a norm with smaller chromatic number.

To a reader familiar with the interplay between extremal and
probabilistic combinatorics, a natural first step in approaching an
extremal question is to ask what happens for a typical object. There
does not appear to be a canonical probability distribution on the set
of norms on $\mathbb R^d$, so instead we use the standard
category-theoretic notion of typicality. We consider the space of norms
on $\mathbb R^d$, identified with their compact, convex,
origin-symmetric unit balls, and equip it with the Hausdorff metric.
This space is a Baire space. Recall that a set is \emph{nowhere dense}
if every nonempty open set contains a smaller nonempty open set
disjoint from it, and that a set is \emph{meagre} if it is a countable
union of nowhere dense sets. The Baire category theorem says that a
meagre set cannot contain a nonempty open set, or equivalently, that a
countable intersection of dense open sets is dense. We say that a
property holds for a \emph{typical} norm if the set of norms for which
it fails is meagre. Thus, for readers accustomed to probabilistic
language, ``typical'' here plays the role of ``almost surely'', with
meagre exceptional sets playing the role of null sets. In particular,
countably many properties that each hold for a typical norm hold
simultaneously for a typical norm; indeed, every nonempty open set of
norms contains a norm satisfying all of them. This is the
category-theoretic analog of discarding countably many probability-zero exceptional events, and it gives the same kind of existence conclusion that is familiar from the probabilistic method, without requiring a probability measure. For our upper bound, we will use only one rather simple property of typical norms, established in \cite{ABS}; we refer the interested reader to that paper for a more detailed introduction and further references.

Alon, Buci\'c, and Sauermann \cite{ABS} proved that for a typical norm the chromatic number is at most an exponential $2^d$ and observed this is tight for $d=2$. They also raise the very natural question of how close this is to being tight. We answer this question completely, and show that the answer is in fact linear in $d$.

\begin{thm}\label{thm:main}
For every positive integer $d$, the unit-distance graph of a typical norm $\|\cdot\|$ on $\R^d$ satisfies
\[
    \chi(G_{\|\cdot\|})\le 2d.
\]
\end{thm}

We also show that this upper bound is best possible.

\begin{thm}\label{thm:lower-bnd}
For every positive integer $d$, there exists a nonempty open set of norms $\|.\|$ on $\R^d$ for which the unit-distance graph satisfies
$$
    \chi(G_{\|\cdot\|})\ge 2d.$$
\end{thm}

Since in a Baire space every comeager set (complement of a meager set) intersects every open set, we conclude that it is not possible to improve \Cref{thm:main}, and in particular that there exists a norm on $\|.\|$ on $\R^d$ for which $\chi(G_{\|\cdot\|})=2d$. To the best of our knowledge this is the first result which determines exactly the Hadwiger Nelson number of a strictly convex norm in $\R^d$ (strict convexity is a classical property of a typical norm \cite{klee}), this was an open problem of Robertson from 1990's (see \cite[Problem 5]{Chilakamarri}) in the planar case $d=2$, up until the recent result of \cite{ABS}.

In \cite{ABS} the authors point out that their result in two dimensions, combined with the aforementioned breakthrough in the Euclidean case \cite{Exoo-Ismailescu,de-Grey}, showcases that the Euclidean planar norm is atypical since the chromatic number of the planar Euclidean unit distance graph is at least $5$ while the typical one is $4$. Our result shows a much more stark contrast, where the answer in large dimensions is known to be exponential in the Euclidean case, while being linear in the typical one.

\subsection{View obstruction and lattice zonotopes}\label{subsec:view-obstruction-intro}

A slight variant of the key ingredient in the proof of \Cref{thm:main} is a very appealing high-dimensional, matrix generalization of the infamous Lonely runner conjecture. In order to state it, it will be convenient to write $\normZ{t}=\dist(t,\Z)$ for the distance of $t$ to the closest integer.

\begin{thm}\label{thm:affine-avoidance}
Let $1\le k<n$, let 
$\mathbf{a_1},\ldots,\mathbf{a_n} \in \R^k$, be such that any $k$ of them are linearly independent. Let also $\mathbf b\in\R^n$. Then, 
\begin{equation}\label{eq:affine-avoidance-intro}
    \sup_{\mathbf x\in\R^k}\min_{i\in[n]}
    \normZ{\mathbf{a_i} \cdot \mathbf x-b_i}
    \ge \frac{k}{2n}.
\end{equation}
If all $\mathbf{a_i}$ have rational coordinates, the supremum is attained.
\end{thm}

Geometrically speaking, if we write $A$ for the matrix with rows $\mathbf{a_i}$, \Cref{thm:affine-avoidance} says that the closure of every affine $k$-flat
\[
   \{A\mathbf x-\mathbf b+\Z^n:\mathbf x\in\R^k\}\subseteq\R^n/\Z^n
\]
which is in general position with respect to the coordinate directions, meets the central cube
\[
   \left[\frac{k}{2n},1-\frac{k}{2n}\right]^n.
\]
So in some sense, this central cube ``obstructs the view'' to infinity. This puts it in the setting of general view obstruction problems, where one attempts to grow a certain object centered at every grid cell (in say $\Z^d$) periodically translated, and understands how large these objects need to be to ensure that every line (or higher dimensional subspace) from the origin intersects them (and hence blocks their view of infinity). We point the reader to \cite{VO-BieniaEtAl1998} and its references for more details. 

This result can also be viewed as a shifted, multi-parameter relative of the Diophantine-approximation problems of Wills and Cusick and of the lonely runner problem
\cite{VO-Wills1967,VO-Wills1968I,VO-Wills1968II,VO-Cusick1973,VO-Cusick1974,VO-BieniaEtAl1998,VO-PerarnauSerra}.

Schoenberg in 1978 formulated the corresponding question in terms of periodic slab families, which he called \emph{monochromes}, and in terms of reflected higher-dimensional billiard signals in a cube
\cite{VO-KonigSzucs,VO-SchoenbergI,VO-SchoenbergII,VO-SchoenbergIII}.
A monochrome of
density $\delta$ is
\[
    M_i(\delta)=\{\mathbf x\in\R^k:\normZ{\mathbf{a_i} \cdot \mathbf x-b_i}\le\delta/2\}.
\]
In other words, this is simply a $\delta$ thickening of the parallel affine hyperplanes given by $\mathbf{a_i} \cdot \mathbf x=b_i+\Z$. 

He referred to an $n$-tuple of such monochromes as \emph{admissible} if any $k$ among the $\mathbf{a}_i$ are linearly independent (this property is sometimes referred to as ``full spark''). 
Equation~\eqref{eq:affine-avoidance-intro} immediately implies that no
admissible family of common density $\delta<k/n$ covers $\R^k$. This proves Schoenberg's view-obstruction conjecture \cite[Conjecture 1]{VO-SchoenbergIII}, arguably one of the most classical open problems in this area.   
Further work on view obstruction and its higher-dimensional variants includes
\cite{VO-Cusick1982,VO-CusickPomerance1984,VO-Chen1994,VO-ChenCusick1999,VO-Dienst1998}.

The same theorem also settles an essentially equivalent covering-radius conjecture of Henze and Malikiosis \cite{VO-HenzeMalikiosis}.
If $\mathbf z_1,\ldots,\mathbf z_m\in\Z^d$ span $\R^d$, write
\[
    Z=\sum_{i=1}^m\left[-\frac{\mathbf z_i}{2},\frac{\mathbf z_i}{2}\right]
\]
for the associated \emph{centred lattice zonotope}, and let
\[
    \mu(Z,\Z^d)=\inf\{\rho>0:\rho Z+\Z^d=\R^d\}
\]
be its \emph{covering radius}.  The generators are in \emph{linear general
position} if every $d$ of them are linearly independent. 
Henze and Malikiosis \cite{VO-HenzeMalikiosis} conjectured that the covering radius of a lattice zonotope generated by $m\ge d$ vectors in
linear general position is at most $d/m$. We confirm this conjecture.

\begin{thm}\label{thm:zonotope-covering}
Let $Z\subseteq\R^d$ be a lattice zonotope generated by $m\ge d$ vectors in
linear general position.  Then
\[
    \mu(Z,\Z^d)\le \frac{d}{m}.
\]
\end{thm}

Henze and Malikiosis identified this estimate as the natural covering radius
form of view obstruction and the lonely runner problems
\cite[Conjecture~1.7]{VO-HenzeMalikiosis}.

\section{Upper bound on the chromatic number of a typical norm}

\subsection{Preliminaries}

The following structural input, see \cite[Sections~5--6]{ABS}, is the only property of a typical norm that we will use. We note that this was first shown by Matou\v{s}ek \cite{matousek} for $d=2$, but his argument does not extend to higher dimensions (which was done in \cite{ABS}).

\begin{prop}[Alon--Buci\'c--Sauermann]\label{prop:ABS}
For every $d\ge2$, let $\|.\|$ be a typical norm on $\R^d$. If $\mathbf{u_1},\dots,\mathbf{u_m}$ are unit vectors in $\|.\|$ with $\mathbf{u_i}\ne\pm \mathbf{u_j}$ for $i\ne j$, then any set of $s$ of these unit vectors spans over $\Q$ at most $ds$ other vectors in the list.
\end{prop}


We will use this as the input for the following standard form of Edmonds' matroid partition
theorem \cite{Edmonds}, specialized here to the vector space case.

\begin{prop}
\label{prop:partition}
Let $E$ be a finite index set, let $(\mathbf{v_e})_{e\in E}$ be a family of
vectors over a field $K$, and let $q\ge1$ be an integer. There is a
partition
\[
    E=I_1\sqcup\cdots\sqcup I_q
\]
such that $(\mathbf{v_e})_{e\in I_s}$ is linearly independent for every $s$
if and only if for every $S\subseteq E$
\begin{equation}\label{eq:partition}
    |S|\le q\,\dim_K\Span_K\{v_e:e\in S\}.
\end{equation}
\end{prop}

Applying this criterion over $\Q$ with $q=d$, \Cref{prop:ABS} implies that every finite family of pairwise
non-antipodal unit directions can be partitioned into $d$
rationally independent families. 


\subsection{Coloring the unit-distance graph}
In this subsection, we will prove \Cref{thm:main}, modulo an auxiliary matrix lemma which we will prove in the subsequent subsection.

\begin{proof}[Proof of \Cref{thm:main}]
Fix a norm $\|.\|$ satisfying Proposition~\ref{prop:ABS}. By the graph-coloring compactness theorem of de Bruijn and Erd\H{o}s \cite{deBruijn_Erdos}, it suffices to $2d$-color every finite subgraph $G$ of its unit-distance graph. If $G$ has no edges, there is nothing to prove. If $G$ is disconnected we can repeat the below proof for each connected component so let us also assume that $G$ is connected.

Let $U=\{\mathbf{u_1},\dots,\mathbf{u_m}\}$ be the family of unit vectors giving rise to the edges of $G$, where for edges arising from parallel vectors we only take one vector. 
Propositions~\ref{prop:ABS} and~\ref{prop:partition}, applied over
$\Q$, give a partition
\[
    U=U_1\sqcup\cdots\sqcup U_d
\]
with every $U_s$ rationally independent. Assume also that $U_1$ was chosen to be maximal subject to this condition (simply move elements to it from the other sets so long as they preserve this condition), so that each vector in $U_2\sqcup\cdots\sqcup U_d$ is a rational linear combination of those in $U_1$. Let $W:=\Span_{\Q}U$ and $N:=\dim_{\Q}W.$ Let us also extend each $U_i$ for $i \ge 2$ into a basis for $W$, by duplicating and adding some vectors from $U_1$ (using the matroid exchange axiom) in order to make each $U_i$ a basis for $W$.


We will change basis to $U_1$, with the above maximality assumption allowing us to express every vector in $U_2\sqcup\cdots\sqcup U_d$ as a linear combination with rational coefficients of the basis vectors in $U_1$. In other words, each of our unit vectors in $U$ has all coordinates rational in the basis $U_1$. If we translate $G$ so that one of its vertices falls to the origin, we know by the connectivity assumption that every other vertex also has rational coordinates in our new basis. Indeed, there is a path in $G$ joining this vertex to the one at the origin showing its coordinate vector is an integral linear combination of vectors in $U$, and hence also has rational coordinates in the basis $U_1$. We may rescale the basis so that all of these coordinates become integral. 

Our plan now is to very carefully pick a (real) vector $\mathbf{x}$ with the property that its dot product with each of our unit vectors in this final coordinate system falls far from an integer (this will be achieved by Lemma \ref{lem:matrix}), namely, far here will mean at least $\frac1{2d}$. Finally, we color each of the vertices of $G$ by computing the dot product of its coordinate vector in this final basis with $\mathbf{x}$ and give it color $r$ if this dot product is in 
$$\left[\frac{r}{2d},\frac{r+1}{2d}\right) \bmod \mathbb{Z}.$$
We note that this gives a proper coloring since two adjacent vertices in $G$ mapping into the same interval would imply that the unit vector they define has a dot product with $\mathbf{x}$ which is of distance strictly smaller than $\frac1{2d}$ from an integer. 

It remains to show we can pick an $\mathbf{x}$ as above. Let us write $A_1,\ldots,A_d$ for the matrices with rows of $A_i$ being the coordinate vectors of vectors in $U_i$ in our final basis above. So in particular, each $A_i$ is integral and invertible (given each $U_i$ is a basis for $W$). 
\Cref{lem:matrix}, which we will prove in the subsequent subsection, shows that these conditions are precisely sufficient to guarantee the desired separation.




This lemma, applied with our matrices $A_s$, $q=d$, and $\mathbf{b_1}=\ldots=\mathbf{b_q}=\mathbf{0}$ (so zero shifts\footnote{We need the shifts in the statement to help with the induction proof of the lemma.}), gives $\mathbf{x}\in\R^N$ such that for each $s$
\[
    A_s \mathbf{x} \in \left[\frac{1}{2d},1-\frac{1}{2d}\right]^N+ \Z^N.
\]
In particular, each coordinate of $A_s \mathbf{x}$ is at least $\frac1{2d}$ away from any integer, as required by the above argument. This completes the proof of \Cref{thm:main} (modulo the proof of Lemma~\ref{lem:matrix}).
\end{proof}

\subsection{The auxiliary matrix integer avoidance lemma}




In this subsection, we prove the final missing ingredient in our proof of \Cref{thm:main}. We note that this is a variant of \Cref{thm:affine-avoidance}, where we do not get to assume that every set of $k$ vectors is linearly independent, but only that their set partitions into such, on the other hand, we are working over the integers, which allows us to avoid the supremum. 

\begin{lem}\label{lem:matrix}
Let $q,N$ be positive integers. Let $A_1,\dots,A_q$ be integer $N \times N$ invertible matrices.
For arbitrary $\mathbf{b_1},\dots,\mathbf{b_q}\in\R^N$, there exists $\mathbf{x}\in\R^N$ such that for every $1\le s\le q$
    \begin{equation}\label{eq:target}
        A_s\mathbf{x}-\mathbf{b_s} \in \left[\frac{1}{2q},1-\frac{1}{2q}\right]^N+ \Z^N.
    \end{equation}
\end{lem}

\begin{proof}
For $q=1$, using that $A_1$ is invertible we may simply take
$\mathbf{x}=A_1^{-1}(\mathbf{b_1}+\tfrac12\mathbf1)$. Henceforth assume $q\ge2$.

We induct on $N$. For the base case, when $N=1$, since $A_s \in \mathbb{Z}$ for all $s$, we know that changing $x$ by an integer does not change whether it satisfies \eqref{eq:target} for any given $s$. So the subset of $[0,1]$ consisting of $x$ which fail \eqref{eq:target} for any given $s$ is a finite union of open intervals of total length $1/q$. Given we are taking $q$ such intervals, there exists a choice of $x \in [0,1]$, for which \eqref{eq:target} holds for all $s$ simultaneously, as desired.\footnote{We note that the case $N=1$ is also going to be covered by the second case we will consider below, which does not invoke the induction hypothesis, but this short argument is quite illustrative of a number of ideas we will use below.}  

Let $m=qN$, and let $\mathbf{a_1^{\mathsf T}},\ldots, \mathbf{a_m^{\mathsf T}}$ denote the row vectors of our matrices and $\beta_1,\ldots, \beta_m \in \R$ the corresponding coordinates of the shift vectors. By the invertibility of our matrices these row vectors are partitioned into $q$ bases of $\R^N$. 
Therefore, any subset of $s$ of our vectors has rank at least $s/q$. We call such a subset \emph{tight} if it has rank exactly $s/q$.



\medskip
\noindent\textbf{Case 1:} a proper tight subset exists.

Suppose that $U$ is a proper tight subset of our row vectors with rank $r$. Since $U$ is nonempty we have $r>0$ and since tightness implies $|U|=rq$ we must also have $r<N$, as otherwise $|U|=qN$ would contain all our row vectors. Since $U$ has rank $r$ each original basis has at most $r$ vectors in $U$, so $|U|=rq$ forces each basis to have exactly $r$ vectors in $U$.

In particular, $\langle U \rangle$ has a basis consisting of $r$ integer vectors. We next choose a change of basis matrix $P$ whose first $r$ columns are these $r$ vectors and the last $N-r$ columns are an integer basis for $\langle U\rangle^\perp.$ Such a matrix exists since $\langle U \rangle$ has an integer basis so running the standard Gaussian elimination algorithm produces a rational basis for $U^\perp$, which we can scale to make it integral. 
Note that $P$ is guaranteed to be nonsingular.

After permuting the rows of each $A_s$ and the corresponding coordinates of the shift vectors $b_s$ to place the $r$ row vectors of $A_s$ making a basis for $\langle U \rangle$ at the start, we can write
\[
    A_sP=
    \begin{pmatrix}
        C_s&0\\
        D_s&F_s
    \end{pmatrix},
    \qquad
    \mathbf{b_s}=\begin{pmatrix}
        \mathbf{b_s'}\\
        \mathbf{b_s''}
    \end{pmatrix}
\]
where $C_s$ is an $r \times r$ nonsingular integer matrix and $F_s$ is an $N-r \times N-r$ nonsingular integer matrix, $\mathbf{b_s'}\in \R^r$ and $\mathbf{b_s''}\in \R^{N-r}$. By induction we can choose $\mathbf{y}\in \R^r$ such that every coordinate
of $C_s\mathbf{y}-\mathbf{b_s'}$ has the required separation.
With this $\mathbf{y}$ fixed, we want to choose $\mathbf{z}\in \R^{N-r}$ such that 
$F_s\mathbf{z}+D_s\mathbf{y}-\mathbf{b_s''}$ has the required separation. This can once again be ensured by invoking induction, on the invertible integral matrix $F_s$, but with modified shift vector $\mathbf{b_s''}-D_s\mathbf{y}$, which gives us $\mathbf{z} \in \R^{N-r}$ for which indeed $F_s\mathbf{z}-(\mathbf{b_s''}-D_s\mathbf{y})$ has the required separation. Finally, taking $\mathbf{x}=P\begin{pmatrix}
        \mathbf{y}\\
        \mathbf{z}
    \end{pmatrix},$ gives the desired $\mathbf{x}$ in this case.

\textbf{Case 2.} there is no proper nonzero tight subset.

The case assumption guarantees that any proper subset $U$ of our row vectors with rank $r$ has size at most $qr-1$. We will exploit
this one unit of slack to construct positive weights that balance
the influence of all the row vectors equally in a certain global optimization problem attempting to pick a vector $\mathbf{x}$ which is as close to a half-integer as possible.

Our plan now is to consider a weighted sum of squares of distances to half integers function to help us choose our $\mathbf{x}$. The idea is that we wish to pick an $\mathbf{x}$ for which $A_s \mathbf{x}-\mathbf{b_s}$ is as close to a point in the lattice $\tfrac12\mathbf1 + \Z^N$ as possible, \emph{simultaneously} for all $s$. 
To do so we consider the function
\[
    \Phi(\mathbf{x},\mathbf{k})=\sum_{i\in [m]}w_i
        \left(\mathbf{a_i^{\mathsf T}}\mathbf{x}-\beta_i-\frac12-k_i\right)^2,
\]
where $,\mathbf{w} \in \R^m, \mathbf{x}\in\R^N, \mathbf{k}\in\Z^m$. Here, one should think of the ``half-integer'' vector $\tfrac12\mathbf1+\mathbf{k}$ as the target vector and $\mathbf{w}$ as a vector of weights, which we need to pick carefully in order to ``normalize'' between different row vectors (keep in mind that our row vector might have wildly differing norms, and hence some might have a disproportionate influence on the function).

Before proceeding to show how we pick our weights function let us note that once we fix the weigths, we will pick $\mathbf{x}$ and $\mathbf{k}$ to minimize $\Phi$ across all $\mathbf{x}\in\R^N, \mathbf{k}\in\Z^m$. That such a minimum exists follows since if we fix $\mathbf{x}$ each coordinate of $\mathbf{k}$ only impacts its own term and minimizes it if $k_i$ is the closest integer to $\mathbf{a_i^{\mathsf T}}\mathbf{x}-\beta_i-\frac12$. So, 
$$\min_{\mathbf{k}\in\Z^m}\Phi(\mathbf{x},\mathbf{k})=\sum_{i\in [m]}w_i
        \dist \left(\mathbf{a_i^{\mathsf T}}\mathbf{x}-\beta_i-\frac12, \Z\right)^2.$$
Now viewing this as a function in $\mathbf{x}$, it is clearly continuous and $\Z^N$-periodic function of $\mathbf{x}$ (since each $\mathbf{a_i}$ is integral), hence it attains its (global) minimum on the compact set $[0,1]^N$.
Armed with this minimizer $(\mathbf{x}, \mathbf{k})$ we are going to show that if \eqref{eq:target} fails for this $\mathbf{x}$ then we can make a local modification and decrease the function giving us a contradiction.   

The first step is to choose the weights. We will pick them carefully as a solution to a certain optimization problem (chosen with hindsight given the above plan, and can also alternatively be obtained by an application of the much more general result from \cite{VO-Barthe}). We will denote the variable for the optimization problem by $\mathbf{t}=(t_i)_{i\in [m]}\in\R^m$, and the optimization problem we will consider is to minimize the function $F$ defined by
\[
    M(\mathbf{t})=\sum_{i=1}^m e^{t_i}\mathbf{a_i}\mathbf{a_i^{\mathsf T}},
    \qquad
    F(\mathbf{t})=\log\det M(\mathbf{t})-\frac1q\sum_{i=1}^mt_i.
\]
It will come in useful that the matrix $M(\mathbf{t})$ is positive definite (in order to ensure $F$ is continuous). The correction/normalization term in the definition of $F$ was chosen so that
\begin{equation}\label{eq:invariance}
    F(t+c\mathbf1)=F(t)\qquad(c\in\R).
\end{equation}
This follows by recalling that $m=qN$ and by simply plugging the RHS into the above definition. This equality allows us to restrict our minimization of $F$ to only minimize over the hyperplane $H=\{\mathbf{t}:\sum_{i=1}^m t_i=0\}$.

Our next goal is to show that $F(\mathbf{t})\to \infty$ as
$\|\mathbf{t}\|_\infty\to\infty$ within $H$, in order to be able to restrict attention to a compact subset over which we can find a global minimum. This allows us to be rather wasteful in the following argument.

Let $\mathcal B \subseteq \binom{[m]}{N}$ be the collection of sets $B \subseteq [m]$, of size $N$, such that $\{\mathbf{a_i^{\mathsf T}} \mid i \in B\}$ is a basis  of $\mathbb{R}^N$. For each such $B$ we let $A_B$ be the square matrix with rows given by $\{\mathbf{a_i^{\mathsf T}} \mid i \in B\}$.
Cauchy--Binet formula gives
\[
    \det M(\mathbf{t})
    =\sum_{B\in\mathcal B}
      (\det A_B)^2\exp\left(\sum_{i\in B}t_i\right),
\]
where we used that $M(\mathbf{t})=R(\mathbf{t})^{\mathsf T}R(\mathbf{t})$, for $R$ being an $m \times N$ matrix with $i$-the row given by $e^{t_i/2}\mathbf{a_i^{\mathsf T}}$. We also used, that for any choice of $N$ rows of $R$ which are not linearly independent the determinant vanishes, so the above sum may indeed only sum over the bases.

The determinants of $A_B$ are nonzero integers for $B\in\mathcal B$ implying $(\det A_B)^2\ge 1$. So, by being rather wasteful we get that for every $B\in\mathcal B$, 
\begin{equation}\label{eq:basisbound}
    F(\mathbf{t})=\log\det M(\mathbf{t})-\frac1q\sum_{i=1}^mt_i\ge\sum_{i\in B}t_i-\frac1q\sum_{i=1}^m t_i.
\end{equation}

Reorder the indices so that $t_1\ge\cdots\ge t_m$, and let  $r_k$ be the rank of $\{\mathbf{a_i} \mid i \in [k]\}$. Now process the vectors $\mathbf{a_1},\ldots, \mathbf{a_m}$ one at a time with greedily accepting a vector
precisely when it increases the span of the current set. The accepted vectors among
the first $k$ vectors hence span all of those first $k$ vectors, so the final
basis $B$ satisfies $|B\cap [k]|=r_k$. Summation by parts yields

\[
\begin{aligned}
    \sum_{i\in B}t_i=\sum_{i\in B} \left(t_m+\sum_{k=i}^{m-1} (t_k-t_{k+1})\right)
       &=Nt_m+\sum_{k=1}^{m-1}(t_k-t_{k+1})r_k,\\
    \frac1q\sum_{i=1}^m t_i=\frac1q\sum_{i=1}^m \left(t_m+\sum_{k=i}^{m-1} (t_k-t_{k+1})\right)
       &=Nt_m+\sum_{k=1}^{m-1}(t_k-t_{k+1})\frac{k}{q}.
\end{aligned}
\]
Since by the case assumption there are no tight proper subsets we know that $qr_k-k\ge 1$.
Consequently \eqref{eq:basisbound} gives
\begin{equation}\label{eq:coercive}
\begin{aligned}
    F(t)
    &\ge\sum_{k=1}^{m-1}(t_k-t_{k+1})
               \left(r_k-\frac{k}{q}\right)\ge\frac{t_1-t_m}{q}.
\end{aligned}
\end{equation}
On $H$, we have $t_m\le0\le t_1$, so
$t_1-t_m\ge\|\mathbf{t}\|_\infty$. Thus $F(\mathbf{t})\to\infty$ as
$\|\mathbf{t}\|_\infty\to\infty$ within $H$. It follows that $F$ attains a
minimum $\mathbf{t^*}\in H$. By \eqref{eq:invariance}, this is also a global minimum on $\R^m$.

We are now ready to define our weights. Put, for each $i \in [m]$
\[
    w_i:=e^{t_i^*},\qquad
    M:=M(\mathbf{t}^*)=\sum_i w_ia_ia_i^{\mathsf T},\qquad
    \lambda_i:=w_ia_i^{\mathsf T}M^{-1}a_i.
\]
The next ingredient is to show that the ``balancing factors'' $\lambda_i$ given above are all equal to $1/q$.

To see this, let $i \in [m]$, the function 
    $$f_i(u):=F(\mathbf{t^*}+u \mathbf e_i)-F(\mathbf{t^*})= \log \frac{\det M(\mathbf{t^*}+u \mathbf e_i)}{\det M} - \frac{u}{q}=\log(1+(e^u-1)\lambda_i)-   \frac{u}{q},$$
where to compute the ratio of determinants we used the matrix determinant lemma giving us 
\[
    \det\bigl(M+(e^u-1)w_i\mathbf{a_i}\mathbf{a_i^{\mathsf T}}\bigr)
    =\det(M)\bigl(1+(e^u-1)\lambda_i\bigr).
\]
Using that $\mathbf{t}^*$ is a global minimum of $F$ we know that $f_i$ has  a minimum at $u=0$ so 

\begin{equation}\label{eq:balanced}
0=\frac{d f_i}{d u}(0)
      =\lambda_i-\frac1q \implies \lambda_i=\frac1q.
\end{equation}

We are now ready for the final part of the proof. We return to our function
$$\Phi(\mathbf{x},\mathbf{k})=\sum_{i\in [m]}w_i
        \left(\mathbf{a_i^{\mathsf T}}\mathbf{x}-\beta_i-\frac12-k_i\right)^2,$$
which is now fully defined. As argued above there exists a global minimizing pair $(\mathbf{x},\mathbf{k})$. Let
\[
    r_i=\mathbf{a_i^{\mathsf T}}\mathbf{x}-\beta_i-\frac12-k_i.
\]
Each $k_i$ is a nearest integer to
$\mathbf{a_i^{\mathsf T}}\mathbf{x}-\beta_i-\frac12$, so $|r_i|\le \frac12$. Our target is to show $|r_i|\le \frac12-\frac{1}{2q}$, for all $i\in[m]$, which ensures the desired separation. 

Keeping $\mathbf{k}$ fixed and differentiating $\Phi$ with respect to $\mathbf{x}$ gives
\begin{equation}\label{eq:normal}
    \sum_{i=1}^m w_ir_ia_i=0.
\end{equation}
Indeed, fix $\mathbf{v} \in \R^N$ and note that $$f(s):=\Phi(\mathbf{x}+s\mathbf{v},\mathbf{k})=\sum_{i=1}^m w_i(r_i+s\mathbf{a_i^{\mathsf T}}\mathbf{v})^2.$$
Since $$0=\frac{df}{ds}(0)=2\mathbf{v^{\mathsf T}}\sum_{i=1}^m w_ir_ia_i,$$ and since $\mathbf{v}$ was arbitrary, the conclusion follows.

Fix $j\in [m]$, put $\mathbf{h}=w_jM^{-1}\mathbf{a_j}$, and let $\mathbf e_j$ be the
$j$-th coordinate vector of $\R^m$. For $\sigma\in\{\pm1\}$,
\eqref{eq:normal} gives
\[
\begin{aligned}
 &\Phi(\mathbf{x}+\sigma \mathbf{h},\mathbf{k}+\sigma\mathbf e_j)-\Phi(\mathbf{x},\mathbf{k})=-2\sigma w_jr_j+\mathbf{h^{\mathsf T}}M\mathbf{h}
                   -2w_j\mathbf{a_j^{\mathsf T}}\mathbf{h}+w_j=w_j(1-\lambda_j-2\sigma r_j).
\end{aligned}
\]
Minimality makes this nonnegative for either choice of the sign $\sigma$. Therefore,
\[
    |r_j|\le\frac{1-\lambda_j}{2}=\frac{q-1}{2q},
\]
as desired.
\end{proof}

\textbf{Remark.} Let us touch briefly upon the aforementioned question of how one could find the weights given the idea to use the weighted
least-squares argument above.  Given positive weights \(w_i\), put
\[
    M(\mathbf{w})=\sum_i w_i \mathbf{a_i}\mathbf{a_i^{\mathsf T}}
    \qquad\text{and}\qquad
    \lambda_i(\mathbf{w})
    =
    w_i\mathbf{a_i^{\mathsf T}}M(\mathbf{w})^{-1}\mathbf{a_i}.
\]
The preceding argument produces a point \(\mathbf{x}\) satisfying
\[
    \bigl\|\mathbf{a_i^{\mathsf T}}\mathbf{x}-\beta_i\bigr\|_{\mathbb R/\mathbb Z}
    \geq \frac{\lambda_i(\mathbf{w})}{2}.
\]
Thus, to obtain the best uniform estimate from this method, one is
naturally led to make the quantities \(\lambda_i(\mathbf{w})\) as equal as
possible.  Since
\[
    \sum_i\lambda_i(\mathbf{w})
    =
    \operatorname{tr}\!\left(
        M(\mathbf{w})^{-1}\sum_i w_i\mathbf{a_i}\mathbf{a_i^{\mathsf T}}
    \right)
    =n
\]
and there are \(dn\) rows, their average is \(1/d\).  The ideal choice
 of weights should therefore satisfy $\lambda_i(\mathbf{w})=\frac1d$ for every $i$. The log-determinant functional arises precisely from trying to solve these equations.  

\section{Lower bounding the chromatic number of certain norms}

In this section, we prove our lower bound (\Cref{thm:lower-bnd}), establishing that there is an open set of norms with chromatic number of their unit distance graph being at least $2d$.

The proof has three parts. The first part is concerned with defining a specific target graph $H_d$ with chromatic number $2d$ that we seek to find in the unit distance graphs. The second part is a stability lemma for prescribed differences showing that given a suitable norm in which we can find $H_d$, we can also find it in any slight perturbation of the norm. The third ingredient is concerned with constructing such a suitable norm.

\subsection{Generalised Moser spindle}\label{sec:spindle}
Let $\|.\|$  be a norm on $\R^d$ and suppose we were able to find the following configuration. Consisting of $2d$ points, including two distinguished points $a,b$, such that
\begin{equation}\label{eq:target-distances}
 \|\mathbf{p}-\mathbf{q}\|=1\quad
 \text{for every unordered pair }\{\mathbf{p},\mathbf{q}\}\ne\{\mathbf{a},\mathbf{b}\},
 \qquad \|\mathbf{a}-\mathbf{b}\|=\frac12.
\end{equation}

Given such a configuration $\Pcal$ we can take a centrally symmetric copy of it $2\mathbf{a} - \Pcal$ obtained by reflecting every point in $\Pcal$ across $\mathbf{a}$, giving us the second copy, with the same properties and with $\mathbf{b'}=2\mathbf{a}-\mathbf{b}$ being the image of $\mathbf{b}$, so that $\|\mathbf{b}-\mathbf{b'}\|=2\|\mathbf{b}-\mathbf{a}\|=1$. In particular, since in the unit distance graph both configurations correspond to $K_{2d}$ with an edge removed, if we were able to properly color the points of $\Pcal$ and $2\mathbf{a}-\Pcal$ using less than $2d$ colors this would force both $\mathbf{b}$ and $\mathbf{b'}$ to receive the same color as $\mathbf{a}$, giving us a monochromatic pair at distance one, which is a contradiction. 



Our goal now is to show we can find a configuration $\Pcal$ as in \eqref{eq:target-distances}.
\subsection{A stability lemma}


The following stability lemma says that if we find a $2d$ point configuration in $\R^d$ satisfying the property that the vectors defined by the differences between every pair of the points are distinct vertices of the convex hull polytope they define, then we can find a norm so that every norm in an open neighborhood close to it contains an equilateral point configuration, not far from our original point configuration. In reality, we will also want to allow scaling factors ($s_e$ below), which allow us to find point configurations with prescribed distances  (where, also for the convex hull assumption, we need to consider the pairs scaled by the same factor).

\begin{lem}\label{lem:stability}
Let $\mathbf{P_0},\ldots,\mathbf{P_{2d-1}}\in\R^d$, and assign a number $s_e>0$ to
each edge $e=ij$ of $K_{2d}$.  Orient the edges arbitrarily and put
\[
 \mathbf{u_e}=s_e(\mathbf{P_i}-\mathbf{P_j}).
\]
Suppose the vectors $\pm \mathbf{u_e}$ are pairwise distinct vertices of their
convex hull.  Then there is a norm $\|.\|_0$ with the following
property.  For every $\delta>0$, there is an open neighbourhood
$\Ucal$ of $\|.\|_0$ such that each $\|.\|\in\Ucal$ admits points
$\mathbf{P'_0},\ldots,\mathbf{P'_{2d-1}}$ satisfying
\[
 \mathbf{P'_0}=\mathbf{P_0},\qquad \max_i\norm{\mathbf{P'_i}-\mathbf{P_i}}<\delta,
 \qquad \bigl\| s_e(\mathbf{P'_i}-\mathbf{P'_j})\bigr\|=1
 \quad  \forall e=ij.
\]
\end{lem}

The proof of the lemma exploits the fact that the vertices of the
convex-hull polytope are exposed.  For each vertex, we choose a
supporting hyperplane passing through it and containing the polytope
entirely on one side.  After suitable normalization, the corresponding
slabs define a centrally symmetric polytope, which we take as the unit
ball of our target norm $\|.\|_0$.  We choose the supporting hyperplanes
carefully so that this construction gives a genuine norm, so that each
of the original vertices lies in the relative interior of a unique
facet of the unit ball, and so that the linearisation of the resulting
distance equations is invertible.

Given a norm close to our target norm, we then consider small
perturbations of the original points and define an error map measuring
the failure of the perturbed configuration to satisfy the prescribed
distance equations.  For the target norm, this error map is locally an
invertible linear map, while for a nearby norm, it is a uniformly small
perturbation of that map.  Once these facts have been established, a
standard application of Brouwer's fixed-point theorem provides a small
perturbation for which the error vanishes. This proof strategy is
heavily inspired by the arguments developed by Brass~\cite{Br} and
Dekster~\cite{De} in their work on finding large equilateral sets, as well as by recent work of Greilhuber, Schildkraut, and Tidor \cite{greilhuber2024more} establishing tight lower bounds for the unit distance problem in the typical norm.
We therefore postpone the full proof to Appendix~\ref{sec:appendix}.


\subsection{Finding the initial exposed configuration}
The final missing ingredient is to find $2d$ points in $\R^d$ so that their pairwise difference vectors (suitably scaled) all give exposed vertices in the convex hull polytope they define. 

Fix $d\ge3$ and put $r=d-1$.  Write
$\R^d=\R^{d-1}\times\R$, with coordinates $(\mathbf{x},z)$, with $\mathbf{x} \in \R^{d-1}$ and $z \in \R$. We think of $z$ as the ``height'' of the point and will exploit it to break certain symmetries. 
We let $\mathbf{e_1},\ldots,\mathbf{e_{d-1}}$ be the standard basis of $\R^{d-1}$.  

Let us introduce the following auxilliary parameters
\begin{equation}\label{eq:parameters}
 \varepsilon=\frac1{100},\qquad
 h_1=0,\qquad
 h_i=1+\frac{i-1}{10(d-1)}\quad(2\le i\le r),\qquad
 \mathbf{v}=\frac12\mathbf{e_1}+\varepsilon\sum_{i=2}^r \mathbf{e_i}.
\end{equation}
In particular, note that $1<h_2<\cdots<h_r<\frac{11}{10}.$
Define $2d$ points by
\begin{equation}\label{eq:points}
 \mathbf{p_i^\sigma}=(\sigma \mathbf{e_i},h_i)
 \quad(1\le i\le r,\ \sigma\in\{\pm 1\}),
 \qquad
 \mathbf{q^\tau}=(\tau \mathbf{v},3)
 \quad(\tau\in\{\pm 1\}),
\end{equation}
and let $\Pcal$ be their set. So the $2d-2$ points $p_i^\sigma$ consist of complementary pairs which match the standard basis vectors in the first $d-1$ coordinates, with the first pair being at height $0$ and the remaining ones at height close to $1$ but slightly different ones between the pairs. The final, special pair is very close to being the complementary pair $\frac12 \mathbf{e_1}$, except it got slightly shifted in every other coordinate and put at a much higher height of $3$.

Let $\mathcal U$ be the set of all pair differences of points in $\Pcal$ except for
$\mathbf{q^+}-\mathbf{q^-}$, which we replace by $\mathbf{w}=2(\mathbf{q^+}-\mathbf{q^-})=(4\mathbf{v},\mathbf{0})$.
We note that we include both signs of every vector in $\mathcal U$. Let $K$ be the convex hull of the vectors in $\mathcal U$. The following proposition is the final ingredient needed to invoke \Cref{lem:stability}.

\begin{prop}\label{prop:exposed}
All vectors in $\mathcal U$ are pairwise distinct vertices of the polytope $K$.
\end{prop}

\begin{proof}
We use the following elementary test. Given two points $\mathbf{P},\mathbf{Q}$ in $\mathcal{P}$ if we can find a linear function $\ell:\R^d \to \R$
which has a unique maximum $\ell(P)$ and unique minimum $\ell(Q)$ on
$\Pcal$, then this certifies that the vector $\mathbf{P}-\mathbf{Q}$ is a vertex of $K$ provided also $\ell(P-Q)>|\ell(w)|$. The negative vector is then also exposed by $-\ell$. 

So it suffices to exhibit such a function for each of our vectors in $\mathcal{U}$.

\smallskip
\underline{Opposite pairs.}
For $\mathbf{p_1^+}-\mathbf{p_1^-}$, take $\ell(\mathbf{x},z)=x_1-\frac12x_2.$
Its unique extrema on $\Pcal$ are $1$ and $-1$, at $\mathbf{p_1^+}$ and
$\mathbf{p_1^-}$, while $\ell(\mathbf{w})=2-2\varepsilon<2$.
For $\mathbf{p_i^+}-\mathbf{p_i^-}$ with $i\ge2$, take $\ell(\mathbf{x},z)=x_i$;
its extrema are again uniquely $1,-1$, and $|\ell(w)|=4\varepsilon<2$.
For $\mathbf{w}$ itself, use $\ell(\mathbf{x},z)=x_1+\tfrac12x_2$.  All points of
$\Pcal$ have $\ell$-value in $[-1,1]$, so every pair difference has
absolute value at most $2$, whereas $\ell(\mathbf{w})=2+2\varepsilon>2.$
Thus $w$ and $-w$ are exposed as well.

\smallskip
\underline{Pairs $\mathbf{p_1^\sigma},\mathbf{p_i^\tau}
$ with $i\ge2$.}
Take $\ell(\mathbf{x},z)=-5\sigma x_1+5\tau x_i+z.$ the unique minimum is $-5$ at $\mathbf{p_1^\sigma}$, and the unique maximum
is $5+h_i>6$ at $\mathbf{p_i^\tau}$.  Indeed, the opposite points have
values $5$ and $-5+h_i$, all remaining lower points have value
$h_j\in(1,1.1)$, and $\ell(\mathbf{q^\pm})=3\pm\left(-5\sigma/2+5\varepsilon\tau\right) \in[0.45,5.55].$
Furthermore,
 $|\ell(\mathbf{w})|\le10+20\varepsilon=10.2
 <10+h_i=\ell(\mathbf{p_i^\tau}-\mathbf{p_1^\sigma}).$

\smallskip
\underline{Pairs $\mathbf{p_i^\sigma},\mathbf{p_j^\tau}$ with $i,j\ge2$ and $i\ne j$.}
Let $t=\sgn(h_i-h_j)/100$, and $\ell(\mathbf{x},z)=\sigma x_i-\tau x_j+tz.$
Then, $t(h_i-h_j)>0$.  The unique maximum is
$1+th_i$ at $\mathbf{p_i^\sigma}$, and the unique minimum is
$-1+th_j$ at $\mathbf{p_j^\tau}$.  Among these two pairs of points, the
only other candidates have values $1+th_j$ and $-1+th_i$, so the
strict inequalities follow from the choice of $t$.  Every remaining
point has absolute $\ell$-value at most
$2\varepsilon+3|t|=0.05$.
Finally, $\ell(\mathbf{p_i^\sigma}-\mathbf{p_j^\tau})=2+t(h_i-h_j)>2$, and 
$|\ell(\mathbf{w})|=4\varepsilon|\sigma-\tau|\le8\varepsilon<2.$

\smallskip
\underline{Pairs $\mathbf{q^\tau},\mathbf{p_i^\sigma}$ with $i\ge2$.}
Use $\ell(\mathbf{x},z)=\tau x_1-4\sigma x_i+2z.$
The desired values are $\ell(\mathbf{q^\tau})=\frac{13}{2}-4\sigma\tau\varepsilon\ge6.46,$ and $\ell(\mathbf{p_i^\sigma})=-4+2h_i<-1.8.$
These are the unique maximum and minimum.  The opposite lower
point has value $4+2h_i<6.2$, the other upper point has value
$11/2+4\sigma\tau\varepsilon\le5.54$, the two points $\mathbf{p_1^\pm}$
have values $\pm1$, and all remaining points have values in $(2,2.2)$.  Also,
$\ell(\mathbf{q^\tau}-\mathbf{p_i^\sigma})=\frac{21}{2}-2h_i-4\sigma\tau\varepsilon>8,$ and $|\ell(\mathbf{w})|=|2\tau-16\sigma\varepsilon|<3$.

\smallskip
\underline{Pairs $\mathbf{q^\tau},\mathbf{p_1^\sigma}$.}
The largest last coordinate of a prescribed signed vector is $3$.
Exactly four vectors attain it: $\mathbf{q^\tau}-\mathbf{p_1^\sigma}=(\tau \mathbf{v}-\sigma \mathbf{e_1},3),$ for $\sigma,\tau\in\{-1,1\}.$ Since $\mathbf{v}$ and $\mathbf{e_1}$ are linearly independent, these are the four
vertices of a nondegenerate parallelogram.  Their convex hull is the
face of $K$ supported by $z=3$.  They are therefore vertices of $K$.
Their negatives are the vertices of the opposite face.
\end{proof}

\subsection{Completing the proof of the lower bound}

\begin{proof}[Proof of Theorem~\ref{thm:main}]
Use the configuration \eqref{eq:points}.  Assign multiplier $2$ to
the pair $q^+,q^-$ and multiplier $1$ to every other pair.
Proposition~\ref{prop:exposed} verifies the hypothesis of
Lemma~\ref{lem:stability}.  Thus, for every norm in a sufficiently
small nonempty open family, the points can be perturbed so that all
pairwise distances are $1$, except that the distance between the two
distinguished points is $1/2$. The proof is now complete by the argument given in Section~\ref{sec:spindle}.
\end{proof}



\section{View obstruction and lattice zonotopes}\label{sec:view-obstruction-zonotopes}

In this section, we prove the two applications stated in the introduction.  The argument is a rank-sensitive version of the balancing-and-exchange proof of \Cref{lem:matrix}, without the integrality assumption or the divisibility assumption on the number of vectors. We note, though, that \Cref{lem:matrix} does not require the assumption that \emph{every} $k$ rows are linearly independent, only that the vectors partition into $k$ linearly independent sets.

\subsection{Equal-leverage scaling}
We begin with the normalization underlying the argument. It is essentially a repeat of the argument we used in the proof of \Cref{lem:matrix}, and can be viewed as a special case of the result in \cite{VO-Barthe}.

\begin{lem}\label{lem:equal-leverage}
Let $k < n$ be integers. Let $\mathbf{a_1},\ldots,\mathbf{a_n}\in \R^k$ be such that any $k$ are linearly independent. Then, there are positive
weights $w_1,\ldots,w_n$ such that, for $M=\sum_{i=1}^n w_i\mathbf{a_i}\mathbf{a_i}^{\mathsf T},$
one has for all $i \in [n]$
\begin{equation}\label{eq:equal-leverage}
    w_i\mathbf{a_i}^{\mathsf T}M^{-1}\mathbf{a_i}=\frac{k}{n}.
\end{equation}
\end{lem}

\begin{proof}
For $\mathbf t=(t_1,\ldots,t_n)\in\R^n$, set
\[
    M(\mathbf t)=\sum_{i=1}^n e^{t_i}\mathbf{a_i}\mathbf{a_i}^{\mathsf T},
    \qquad
    F(\mathbf t)=\log\det M(\mathbf t)-\frac{k}{n}\sum_{i=1}^n t_i.
\]
The matrix $M(\mathbf t)$ is positive definite.  Since its size is $k\times k$,
\begin{equation}\label{eq:full-spark-invariance}
    F(\mathbf t+c\one)=F(\mathbf t)
    \qquad(c\in\R).
\end{equation}
We therefore minimize $F$ on the hyperplane
\[
    H=\left\{\mathbf t\in\R^n:\sum_{i=1}^n t_i=0\right\}.
\]

For $I\in\binom{[n]}{k}$, let $A_I$ be the $k\times k$ matrix with row set
$I$.  Cauchy--Binet gives
\begin{equation}\label{eq:full-spark-cauchy-binet}
    \det M(\mathbf t)
    =\sum_{I\in\binom{[n]}{k}}(\det A_I)^2
      \exp\left(\sum_{i\in I}t_i\right).
\end{equation}
Every determinant in this sum is nonzero.  Put
\[
    D=\min_{I\in\binom{[n]}{k}}(\det A_I)^2>0.
\]
After relabelling, suppose that $t_1\ge\cdots\ge t_n$.  Keeping only the term
$I=[k]$ in \eqref{eq:full-spark-cauchy-binet}, we obtain
\[
    F(\mathbf t)\ge \log D+
    \sum_{i=1}^k t_i-\frac{k}{n}\sum_{i=1}^n t_i.
\]
Summation by parts rewrites the last two terms as
\begin{equation}\label{eq:full-spark-coercive}
    \sum_{r=1}^{n-1}(t_r-t_{r+1})
    \left(\min\{r,k\}-\frac{kr}{n}\right).
\end{equation}
Every coefficient in \eqref{eq:full-spark-coercive} is positive; in fact, it is
at least $\min\{k,n-k\}/n$.  Hence
\[
    F(\mathbf t)\ge \log D+
    \frac{\min\{k,n-k\}}{n}(t_1-t_n).
\]
On $H$, the quantity $t_1-t_n$ tends to infinity with
$\|\mathbf t\|_\infty$.  Thus $F$  attains a minimum
$\mathbf t^*$ on $H$.  By \eqref{eq:full-spark-invariance}, this is also a global minimum on $\R^n$.

Set $w_i=e^{t_i^*}$ and $M=M(\mathbf t^*)$.  Differentiating the log
determinant at the minimum gives
\[
    0=\frac{\partial F}{\partial t_i}(\mathbf t^*)
    =w_i\mathbf{a_i}^{\mathsf T}M^{-1}\mathbf{a_i}-\frac{k}{n},
\]
which is \eqref{eq:equal-leverage}.
\end{proof}

\subsection{Simultaneous avoidance of the integers}

\begin{proof}[Proof of \Cref{thm:affine-avoidance}]
Let $A$ be the matrix whose rows are the vectors $\mathbf{a_1}^{\mathsf T},\ldots, \mathbf{a_n}^{\mathsf T}$. Choose the weights from \Cref{lem:equal-leverage}.  For
$\mathbf x\in\R^k$ and $\mathbf z\in\Z^n$, consider the weighted squared
distance to the half-integer lattice
\begin{equation}\label{eq:full-spark-energy}
    \Phi(\mathbf x,\mathbf z)
    =\sum_{i=1}^n w_i
      \left(\mathbf{a_i}^{\mathsf T}\mathbf x-b_i-\frac12-z_i\right)^2.
\end{equation}
Let
\[
    \mu=\inf_{\mathbf x\in\R^k,\,\mathbf z\in\Z^n}
    \Phi(\mathbf x,\mathbf z).
\]
For irrational $A$ this infimum need not be attained, so fix $\eta>0$ and
choose $(\mathbf x_0,\mathbf z)$ with
$\Phi(\mathbf x_0,\mathbf z)\le\mu+\eta$.  Keeping this $\mathbf z$ fixed,
minimize $\Phi(\mathbf{x},\mathbf{z})$ over $\mathbf x$. For this fixed $\mathbf z$, the function
$\Phi(\,\cdot\,,\mathbf z)$ is a strictly convex quadratic with
positive-definite Hessian (equal to $2M$).  It therefore has a unique minimizer
$\mathbf x$, and $\Phi(\mathbf x,\mathbf z)
    \le \Phi(\mathbf x_0,\mathbf z)
    \le \mu+\eta.$ 
    
Write
\[
    r_i=\mathbf{a_i}^{\mathsf T}\mathbf x-b_i-\frac12-z_i.
\]
The normal equations are
\begin{equation}\label{eq:full-spark-normal}
    \sum_{i=1}^n w_ir_i\mathbf{a_i}=0.
\end{equation}

Fix $j\in[n]$ and put
\[
    \mathbf h_j=w_jM^{-1}\mathbf{a_j}.
\]
For $\sigma\in\{-1,1\}$, compare $(\mathbf x,\mathbf z)$ with
$(\mathbf x+\sigma\mathbf h_j,\mathbf z+\sigma\mathbf e_j)$.  Expanding
\eqref{eq:full-spark-energy} and using \eqref{eq:full-spark-normal}, we get
\begin{align*}
 \Phi(\mathbf x+\sigma\mathbf h_j,
       \mathbf z+\sigma\mathbf e_j)-\Phi(\mathbf x,\mathbf z)=-2\sigma w_jr_j
   +\mathbf h_j^{\mathsf T}M\mathbf h_j
   -2w_j\mathbf{a_j}^{\mathsf T}\mathbf h_j+w_j.
\end{align*}
By \eqref{eq:equal-leverage},
\[
    \mathbf h_j^{\mathsf T}M\mathbf h_j
    =w_j\frac{k}{n}
    \qquad\text{and}\qquad
    w_j\mathbf{a_j}^{\mathsf T}\mathbf h_j
    =w_j\frac{k}{n}.
\]
Consequently
\begin{equation}\label{eq:full-spark-exchange}
    \Phi(\mathbf x+\sigma\mathbf h_j,
       \mathbf z+\sigma\mathbf e_j)-\Phi(\mathbf x,\mathbf z)
    =w_j\left(1-\frac{k}{n}-2\sigma r_j\right).
\end{equation}
At every competing pair $\Phi$ takes value at least $\mu$, whereas $\Phi(\mathbf x, \mathbf z) \le \mu+\eta$. Thus the left-hand side of
\eqref{eq:full-spark-exchange} is at least $-\eta$ for both signs $\sigma$, and hence
\begin{equation}\label{eq:full-spark-residual}
    |r_j|\le \frac{n-k}{2n}+\frac{\eta}{2w_j}.
\end{equation}
Let $w_{\min}=\min_i w_i$.  Taking $\eta$ sufficiently small makes the
right-hand side of \eqref{eq:full-spark-residual} less than $1/2$.  We then
have, simultaneously for all $j$,
\[
    \normZ{\mathbf{a_j}^{\mathsf T}\mathbf x-b_j}
    =\normZ{\frac12+r_j}
    =\frac12-|r_j|
    \ge \frac{k}{2n}-\frac{\eta}{2w_{\min}}.
\]
Letting $\eta\to0$ proves \eqref{eq:affine-avoidance-intro}.

If $A$ is rational, choose a positive integer $D$ such
that $DA$ is integral. The function
\[
f(\mathbf x)
=
\min_{i\in[n]}
\normZ{\mathbf a_i^{\mathsf T}\mathbf x-b_i}
\]
is continuous and invariant under translations by $D\mathbb Z^k$.
It therefore attains its maximum on the compact domain
$[0,D]^k$, proving that the supremum is attained. Moreover, the function
\[
\mathbf x\longmapsto
\min_{\mathbf z\in\mathbb Z^n}\Phi(\mathbf x,\mathbf z)
\]
is also continuous and $D\mathbb Z^k$-periodic. Hence its infimum is
attained, so one may take $\eta=0$ in the preceding argument.
\end{proof}

\subsection{Schoenberg's view-obstruction conjecture}
For completeness, we now spell out the deduction in Schoenberg's language. Recall that
\[
    M_i(\delta)=
    \{\mathbf x\in\R^k:\normZ{\mathbf{a_i}^{\mathsf T}\mathbf x-b_i}\le\delta/2\}.
\]
An \emph{$n$-chromo} is a family $(M_i(\delta))_{i\in[n]}$ whose union is
$\R^k$, and it is \emph{admissible} if every $k$ of the normals
$\mathbf a_i$ are linearly independent.  Let $\delta_{k,n}$ be the infimum of the common densities of admissible $n$-chromos in $\R^k$.

\begin{cor}\label{cor:schoenberg}
For every $1\le k<n$,
\[
    \delta_{k,n}\ge \frac{k}{n}.
\]
\end{cor}

\begin{proof}
Suppose that an admissible $n$-chromo has density $\delta<k/n$.  Choose
$\varepsilon>0$ so that $\frac{k}{2n}-\varepsilon>\frac{\delta}{2}.$
By \Cref{thm:affine-avoidance}, there is $\mathbf x\in\R^k$ for which for all $i \in [n]$
\[
    \normZ{\mathbf{a_i}^{\mathsf T}\mathbf x-b_i}
    \ge \frac{k}{2n}-\varepsilon
    >\frac{\delta}{2}.
\]
This point $\mathbf{x}$ belongs to none of the monochromes, contradicting that they cover
$\R^k$.
\end{proof}


\subsection{The covering radius of lattice zonotopes}

\begin{proof}[Proof of \Cref{thm:zonotope-covering}]
Let $B\in\Z^{d\times m}$ be the matrix whose columns are the generators
$\mathbf z_1,\ldots,\mathbf z_m$, so that
\[
    Z=B\left[-\frac12,\frac12\right]^m.
\]

If $m=d$, then $B$ is invertible so $Z+B\mathbb Z^d
    =
    B\bigl([-1/2,1/2]^d+\mathbb Z^d\bigr)
    =
    \mathbb R^d.$
Since $B\mathbb Z^d\subseteq\mathbb Z^d$, it follows that
$Z+\mathbb Z^d=\mathbb R^d$, so the covering radius is at most $1$ as desired.
Hence, assume $m>d$ and put $k=m-d$.

Choose a rational matrix $A\in\Q^{m\times k}$ whose columns form a basis of
$\ker B$.  The matrix $A$ has any $k$ rows linearly independent.  Indeed, suppose the rows indexed by $I\subseteq[m]$, with $|I|=k$, were dependent. This means there exists $\mathbf x \in \R^k \setminus \{\mathbf{0}\}$ such that the submatrix $A_I$ of $A$ obtained by only keeping the rows indexed by $I$ satisfies $A_I\mathbf{x}=\mathbf{0}$. Then, taking $\mathbf{v}=A \mathbf{x}$, we have that $\mathbf v$ is nonzero (since $A$ has full column rank) and $\mathbf v\in\im A=\ker B$ so $B\mathbf v=\mathbf 0$. Furthermore, $A_I\mathbf{x}=\mathbf{0}$ translates to the coordinates of $\mathbf v$ indexed by $I$ being $0$. It is therefore supported on the complementary set $J=[m]\setminus I$, which has size $d$. But then $B_J\mathbf v_J=\mathbf 0,$ contradicting the assumption that every $d$ columns of $B$ are linearly independent. 

Fix $\mathbf y\in\R^d$ and choose $\mathbf c\in\R^m$ with
$B\mathbf c=\mathbf y$.  Apply the rational, attained form of
\Cref{thm:affine-avoidance} to $A$ with the shift $\mathbf b=\frac12\one-\mathbf c$.
There is $\mathbf x\in\R^k$ such that, for every $i\in[m]$,
\[
    \normZ{c_i+(A\mathbf x)_i-\frac12}
    \ge \frac{k}{2m}
\implies 
    \normZ{c_i+(A\mathbf x)_i}
    \le \frac12-\frac{k}{2m}
    =\frac{d}{2m}.
\]
Choose $\mathbf z\in\Z^m$ coordinatewise so that
$\|\mathbf c+A\mathbf x-\mathbf z\|_\infty\le\frac{d}{2m}.$

Using $BA=0$, we obtain
\[
    \mathbf y-B\mathbf z
    =B(\mathbf c+A\mathbf x-\mathbf z)
    \in \frac{d}{m}B\left[-\frac12,\frac12\right]^m
    =\frac{d}{m}Z.
\]
As $\mathbf y$ was arbitrary,
    $\frac{d}{m}Z+B\Z^m=\R^d.$
    
Finally, $B\Z^m\subseteq\Z^d$, so $$\frac{d}{m}Z+\Z^d=\R^d.$$
This is precisely the desired inequality $\mu(Z,\Z^d)\le d/m$.
\end{proof}


\section{Concluding remarks}

Theorem~\ref{thm:lower-bnd} gives one nonempty open family, not a dense family. It therefore does not prove $\chi(G_{\|\cdot\|})=2d$ for a comeagre subset of the entire norm space. In fact, one can give an explicit neighborhood of the Euclidean norm in which the particular spindle used in our proof cannot occur. Let $B_2^d$ denote the Euclidean unit ball, and set
$$
 r_d=
 \begin{cases}
  \displaystyle\sqrt{1+\frac{2}{d}}, & d \text{ even},\\[6pt]
  \displaystyle\sqrt{1+\frac{2}{d-\frac{1}{d+2}}}, & d \text{ odd},
 \end{cases}
 \qquad
 \varepsilon_d=\frac{r_d-1}{r_d+1}.
$$
A theorem of Sch\"utte \cite{Schuette1963} states that every set $S$ of at
least $d+2$ points in $\R^d$ satisfies
\[
 \frac{\max_{\mathbf{x},\mathbf{y}\in S}\|\mathbf{x}-\mathbf{y}\|_2}
 {\min_{\substack{\mathbf{x},\mathbf{y}\in S\\ \mathbf{x}\ne\mathbf{y}}}
  \|\mathbf{x}-\mathbf{y}\|_2}
 \ge r_d,
\]
and that this bound is sharp; see also \cite{Barany1994} for a short proof.
Now let $B$ be the unit ball of a norm $\|\cdot\|$ satisfying which has Hausdorff distance to $B_2^d$ less than $\varepsilon_d$. Let $\varepsilon$ be such that $d_{\mathrm H}(B,B_2^d)<\varepsilon<\varepsilon_d.$ So, $(1-\varepsilon)B_2^d\subseteq B\subseteq(1+\varepsilon)B_2^d.$ Consequently, every pair at unit distance in $\|\cdot\|$ has Euclidean distance in $[1-\varepsilon,1+\varepsilon]$. A unit-distance clique for $\|\cdot\|$ therefore has Euclidean distance ratio at most $\frac{1+\varepsilon}{1-\varepsilon}<r_d,$ and so has at most $d+1$ vertices. Since $2d-1\ge d+2$ for $d\ge3$, every norm in this explicit neighborhood admits no unit-distance $K_{2d-1}$ and hence cannot contain the spindle used in the proof of \Cref{thm:lower-bnd}.

For the Euclidean norm, the known exponential lower bounds imply that the chromatic number is larger than $2d$ in all sufficiently large dimensions \cite{FW}. It is natural to expect strict inequality in every nontrivial dimension. The planar bound $\chi(G_{\|\cdot\|_2})\ge5>4$ \cite{Exoo-Ismailescu,de-Grey} establishes this in two dimensions, explicit constructions also establish it for $d=4,7,8$ (see, \cite{ExooIsmailescuLim2014}, \cite{ExooIsmailescuLim2014}, and \cite{KahleTaha2015}, respectively) and an explicit lower bound due to Cherkashin, Kulikov and
Raigorodskii \cite{CherkashinKulikovRaigorodskii2018} shows that it is in fact true for any $d \ge 9$. This leaves open only the cases $d=3,5,6$. Given this, it could be interesting to provide a simple general argument to show that, in fact, for every integer $d>1$, $\chi(G_{\|\cdot\|_2})>2d.$

We note that it is entirely possible (although in our opinion unlikely) that \Cref{thm:lower-bnd} applies for \emph{all} norms. A conjecture of Petty raised in 1971 (\cite{Pe}, 
see also \cite{BMP-survey}) asserts that the clique number of the unit distance graph\footnote{Usually referred to as the equilateral number.} is at least $d+1$ for every norm $\|.\|$ on $\R^d$. This would imply at least a linear lower bound on the chromatic number for all norms. This is known in dimension $d \le 3$ (\cite{Pe}), and is also known for norms that are sufficiently close to the $d$-dimensional Euclidean norm (\cite{Br}, \cite{De}), or more generally to the $d$-dimensional $\ell_p$-norm $\ell_p^d$ for any $1 <p \leq \infty$ (\cite{SV}). Combining this with a result of \cite{AM} asserting that every $d$-dimensional normed space contains a subspace of dimension $r=e^{\Omega(\sqrt {\log d})}$ which is either close to $\ell_2^r$ or to $\ell_{\infty}^r$, it was proved by Swanepoel and Villa \cite{SV} that the equilateral number of any norm $\|.\|$ on $\R^d$ is at least $e^{b \sqrt {\log d}}$  for some absolute constant $b>0$. This, to the best of our knowledge, provides the best known lower bound which holds for all norms, for the chromatic number problem as well. See \cite{AP}, \cite{Sw0}, \cite{Sw} for more information on the rich history of the study of equilateral numbers. 

We note that as part of the proof of our lower bound \Cref{thm:lower-bnd} we also answer in the negative another question of Alon, Buci\'c, and Sauermann \cite{ABS}. Namely, we show that there is no $0-1$ law for subgraph containment in the unit distance graph of a typical norm. Namely, our argument shows that $K_{2d-1}$ is a subgraph of the unit distance graph of a non-empty open set of norms, but it is also not a subgraph of the unit distance graph of an open set of norms (for any norm close to Euclidean by the argument described above).

We note that a finite field version of the problem tackled by \Cref{lem:matrix} was considered by Nagy, Pach, and Tomon \cite{NPT}. In fact, one can use their main result to prove a weaker separation of $1/d^{O(d)}$ in \Cref{lem:matrix}. Their Conjecture 1.4 would imply a linear separation $1/O(d)$. Given this, it would be very interesting to find a way of using the ideas behind our $1/(2d)$ separation result over $\mathbb{Z}$ to resolve the finite field variant as well.

\textbf{Acknowledgments.} 
This work was initiated at the 2023 Structural Graph Theory workshop held at the IMPAN B\k{e}dlewo Conference Center (Poland) from September 27- October 2, 2023. 
The authors would like to thank Lisa Sauermann, Istvan Tomon, Leo Versteegen, Vanshika Jain, Dima Zakharov, Imre Leader, Luka Mili\'cevi\'c, Filip Ku\v{c}erak, Fankang He, Noah Kravitz, Peter Pal Pach for useful conversations surrounding this problem over the years. ChatGPT 6 Pro has provided us with the proof of the final ingredient we needed in the proof of \Cref{thm:main}, namely that of Lemma \ref{lem:matrix}, following a prolonged discussion, which included sharing various of our observations about the problem. These included the induction idea used in the ``tight'' case and the general approach of trying to minimize a certain distance function to suitably chosen half-integer vectors. The lower-bound argument was also found with the assistance of ChatGPT 6 Pro.

\providecommand{\MR}[1]{}
\providecommand{\MRhref}[2]{%
  \href{http://www.ams.org/mathscinet-getitem?mr=#1}{#2}
}

 \bibliographystyle{amsplain_initials_nobysame}
 \bibliography{ref}

\appendix

\section{Proof of the stability lemma}\label{sec:appendix}
In this section we include a proof of Lemma~\ref{lem:stability}.

\begin{proof}[Proof of Lemma~\ref{lem:stability}]
Write $E=E(K_{2d})$.  For every $e \in E$, the vertex $\mathbf{u_e}$ of the convex hull polytope $K$ of the vectors $\{\pm \mathbf{u_e}\}_{e \in E}$ is exposed, meaning there exists a linear function $\mathbf{x} \to \mathbf{a_e} \cdot \mathbf{x}$ strictly maximized over $K$ at $\mathbf{u_e}$. So, for
each $e\in E$ the set $A_e$ of $\mathbf{a_e} \in \R^d$ such that
\begin{equation}\label{eq:exposing-cones}
 \mathbf{a_e} \cdot \mathbf{u_e}>0,\qquad
 \mathbf{a_e} \cdot \mathbf{u_e}>|\mathbf{a_e} \cdot \mathbf{u_f}|\quad \forall f\ne e
\end{equation}
is nonempty and open.

Given $A=(\mathbf{a_e})_{e \in E} \in (\R^d)^E$ we define a linear map $R_{A}:(\R^d)^{2d-1}\longrightarrow\R^E$, by for each $e=ij$ setting
\begin{equation}\label{eq:Rg} 
(R_{A}(h))_e=s_e \: \mathbf{a_e} \cdot (\mathbf{h_i}-\mathbf{h_j}),
\end{equation}
where we view $h$ as consisting of $2d-1$ vectors $\mathbf{h_1},\ldots,\mathbf{h_{2d-1}}\in \R^d$, and set $\mathbf{h_0}=\mathbf{0}$. These $\mathbf{h}_i$'s will play the role of slight perturbations of our points $\mathbf{P_i}$'s, giving us the desired points for any given norm. In light of this, $\mathbf{h_0}=\mathbf{0}$ captures the fact that our point configurations are translation invariant and simply fixes the point $\mathbf{P_0}$ in place.

We note that these maps are between two spaces of the same dimension, since $d(2d-1)=\binom{2d}{2}=|E|,$ meaning their determinant is defined. The determinant of $R_{A}$ is a polynomial in the coordinates of $\mathbf{a_e}$'s, and we first show that this polynomial is not identically zero.

To see this, we first partition $E=E(K_{2d})$ into $d$ spanning trees
 $T_1,\ldots,T_d$. Such a partition (in fact into Hamilton paths) follows from the classical Walecki's theorem, see  \cite{lucas1892}, which gives an edge decomposition of a clique of odd order into Hamilton cycles and deleting one vertex. 
Since our goal is to show that the determinant polynomial is not identically zero, it is sufficient to find a specific choice $A_0$ for $A$ in such a way that it does not vanish. We chose $A_0$ so that $\mathbf{a_e}=\mathbf{e_k}$ if $e\in T_k$ in our decomposition. To show its determinant does not vanish we will show its kernel is empty. Towards this, assume we have an $h$ as above so that $R_{A_0}(h)=0$. Then, since $s_e>0$ we have $\mathbf{a_e} \cdot (\mathbf{h_i}-\mathbf{h_j})=0$ for all $e\in E$. Focusing on $e \in T_k$, our choice of $A_0$ ensures that the $k$-th coordinates of all the $h_i$ agree along $T_k$.
Since $T_k$ is a spanning tree and $\mathbf{h_0}=\mathbf{0}$, all these coordinates vanish. This holds for every $k$, showing $R_{A_0}$ has an empty kernel and hence a non-zero determinant.

A nonzero real polynomial cannot vanish on a nonempty open set $\prod_{e \in E} A_e$. So, we may therefore choose $A=(\mathbf{a_e})_{e \in E}$ with each $\mathbf{a_e}$ satisfying the corresponding inequalities in \eqref{eq:exposing-cones}, while also making $R_A$ invertible.

Let us normalize by
\[
 \mathbf{b_e}=\frac{\mathbf{a_e}}{\mathbf{a_e} \cdot \mathbf{u_e}}.
\]
Then \eqref{eq:exposing-cones} becomes
\begin{equation}\label{eq:normalized-support}
 \mathbf{b_e} \cdot \mathbf{u_e}=1,\qquad |\mathbf{b_e} \cdot \mathbf{u_f}|<1\quad \forall f\ne e,
\end{equation}
and the map $R_B$ with $B=(\mathbf{b_e})_{e \in E}$ 
\eqref{eq:Rg} remains invertible (each individual scaling corresponds to scaling only one output coordinate of $R$ by \eqref{eq:Rg}, so it does not change its rank).

We are now ready to define our target norm $\|.\|_0$, we simply set
\begin{equation}\label{eq:norm}
 \|\mathbf{z}\|_0=\max_{e\in E}|\mathbf{b_e} \cdot \mathbf{z}|.
\end{equation}
The non-negativity and scaling are immediate, while the subadditivity follows from linearity and triangle inequality, so the only missing ingredient for showing this is indeed a norm is to show that it does not vanish except at $\mathbf{0}$. 
This follows since if $\mathbf{z} \in \R^d$ is a non-zero vector in the orthogonal complement of all the $\mathbf{b_e}$, we could set $\mathbf{h_2}=\mathbf{z}$, and all other $\mathbf{h_i}=\mathbf{0}$ to obtain a nonzero vector in the kernel of
$R_B$.  Thus, \eqref{eq:norm} is a norm.

At $\mathbf{z}=\mathbf{u_e}$, the positive expression $\mathbf{b_e} \cdot \mathbf{z}$ is the unique maximizer in \eqref{eq:norm}, by \eqref{eq:normalized-support}.  Hence
\begin{equation}\label{eq:local-linearity}
 \|\mathbf{z}\|_0=\mathbf{b_e} \cdot \mathbf{z}\qquad\text{for any }\mathbf{z}\text{ in a small neighbourhood of }\mathbf{u_e}.
\end{equation}

For any norm $\|.\|$, define the ``error'' map $F:(\R^d)^{2d-1}\longrightarrow\R^E$, by for each $e=ij$ setting
\[
 F_{\|.\|}(h)_e
 =\bigl\|s_e(\mathbf{P_i}-\mathbf{P_j}+\mathbf{h_i}-\mathbf{h_j})\bigr\|-1,
 \qquad \mathbf{h_0}=\mathbf{0}.
\]
This is chosen so that if we can find a small $h \in (\R^d)^{2d-1}$ for which $F_{\|.\|}(h)=\mathbf{0}$, then the points $\mathbf{P'_i}=\mathbf{P_i}+\mathbf{h_i}$ precisely satisfy the desired conditions of the lemma.

By \eqref{eq:local-linearity}, on a sufficiently small closed
Euclidean ball $B_\rho\subseteq(\R^d)^{2d-1}$ we have
\begin{equation}\label{eq:linear-error}
 F_{\|.\|_0}(h)=R_B(h),
\end{equation}
since $\bigl\|s_e(\mathbf{P_i}-\mathbf{P_j}+\mathbf{h_i}-\mathbf{h_j})\bigr\|_0=\bigl\|\mathbf{u_e}+s_e(\mathbf{h_i}-\mathbf{h_j})\bigr\|_0=\mathbf{u_e} \cdot \mathbf{b_e}+ s_e(\mathbf{h_i}-\mathbf{h_j}) \cdot \mathbf{b_e}=1+s_e(\mathbf{h_i}-\mathbf{h_j}) \cdot \mathbf{b_e}.$ 
The edge vectors in the definition of $F_{\|.\|}$ range over a bounded set as $h$ varies in $B_\rho$.  By homogeneity
of norms, $F_{\|.\|}\to F_{\|.\|_0}$ uniformly on $B_\rho$ as
we take $\|.\|$ arbitrarily close to $\|.\|_0$ in the Hausdorff distance (of their unit balls).  Thus, for all norms $\|.\|$ sufficiently close to $\|.\|_0$,
\[
 \sup_{h\in B_\rho}
 \norm{R_B^{-1}\bigl(F_{\|.\|}(h)-R_B(h)\bigr)}<\rho.
\]
The continuous map
\[
 h\longmapsto h-R_B^{-1}F_{\|.\|}(h)
\]
therefore maps $B_\rho$ into itself.  Brouwer's fixed-point theorem
gives a fixed point $h$, and at that point $F_{\|.\|}(h)=\mathbf{0}$.
Taking $\mathbf{P'_i}=\mathbf{P_i}+\mathbf{h_i}$ proves the lemma.
\end{proof}

\end{document}